\documentclass[a4paper]{article}
\usepackage{minitoc}
\usepackage[titletoc]{appendix}
\usepackage{url}
\usepackage[utf8]{inputenc}
\usepackage[T1]{fontenc}
\usepackage{graphicx}
\usepackage{amsmath}
\usepackage{listings}
\usepackage{moreverb}
\usepackage{eurosym}
\usepackage{fancyvrb}
\usepackage{afterpage}
\usepackage[final]{pdfpages}
\usepackage[nottoc]{tocbibind}
\usepackage{listings}
\usepackage{amsthm}
\usepackage{amsfonts}
\usepackage{geometry}
\usepackage[]{algorithm2e}
\usepackage{fixltx2e}
\usepackage{amssymb}
\usepackage{tikz}
\usetikzlibrary{calc}
\usepackage{float}
\usepackage{color}
\usepackage{lipsum}
\usepackage{multirow}

\newtheorem{theorem}{Theorem}[section]
\newtheorem{lemma}[theorem]{Lemma}
\newtheorem{proposition}[theorem]{Proposition}
\newtheorem{cor}[theorem]{Corollary}
\newtheorem{definition}[theorem]{Definition}

\newtheorem{remark}[theorem]{Remark}

\newtheorem{innercustomthm}{Theorem}
\newenvironment{customthm}[1]
{\renewcommand\theinnercustomthm{#1}\innercustomthm}
{\endinnercustomthm}

\makeatletter
\def\v@rt#1#2{\m@th\ooalign{$\hfil#1|\hfil$\crcr$#1#2$}}
\def\captr{\mathrel{\mathpalette\v@rt\cap}}
\makeatother

\makeatletter

\newcommand{\Rmnum}[1]{\expandafter\@slowromancap\romannumeral #1@}
\makeatother

\newcommand\blfootnote[1]{
	\begingroup
	\renewcommand\thefootnote{}\footnote{#1}
	\addtocounter{footnote}{-1}
	\endgroup
}

\newcommand{\Z}{\mathbb{Z}}
\def \Refl {T}
\def \RRefl {\boldsymbol{T}}
\def \refl {t}
\def \rrefl {\boldsymbol{t}}

\def \SS {{\boldsymbol{\mathcal{S}}}}

\def \ss {{\boldsymbol s}}
\def \tt {{\boldsymbol t}}

\def\bs#1{\boldsymbol{#1}}

\def \cc#1#2#3#4{{\hbox{\sc cc}}(#1,#2,#3,#4)}
\def \tc#1#2#3#4{{\hbox{\sc tc}}(#1,#2,#3,#4)}
\def \PA {\mathrm{PA}}
\def \Sym {\mathrm  {Sym}}
\def \Fix {\mathrm  {Fix}}
\def \fix {\mathrm  {Fix}}
\def \Mov {\mathrm  {Mov}}

\def \im {\mathrm  {im}}
\def \Sym {\mathrm  {Sym}}

\begin{document}

\title{Interval groups related to finite Coxeter groups\\
	Part II}
\author{Barbara Baumeister, Derek F. Holt,  Georges Neaime, and Sarah Rees}
\date{November 24, 2022}
	
\maketitle

\begin{abstract}
We provide a complete description of the presentations of the interval groups related to quasi-Coxeter elements in finite Coxeter groups. In the simply laced cases, we show that each interval group 
	is the quotient of the Artin group associated with the corresponding Carter
diagram by the normal closure of a set of twisted cycle commutators,
one for each 4-cycle of the diagram.
	Our techniques also reprove an analogous result for the Artin
	groups of finite Coxeter groups, which are interval groups corresponding to Coxeter
	elements. We also analyse the situation in the non-simply laced cases, where a new Garside structure is discovered.

Furthermore, we obtain a complete classification of whether the interval group we consider is isomorphic or not to the related Artin group. Indeed, using methods of Tits, we prove that the interval groups of proper quasi-Coxeter elements are not isomorphic to the Artin groups of the same type, in the case of $D_n$ when $n$ is 
	even or in any of the exceptional cases. In \cite{BHNR_Part3}, we show  using different methods that this result holds for type $D_n$ for all $n \geq 4$.
\end{abstract}

\blfootnote{2010 MSC:  20F55, 20F36, 05E15.}
\blfootnote{Keywords and phrases: Coxeter groups, quasi-Coxeter elements, Artin groups, dual approach to Coxeter and Artin groups, generalised non-crossing partitions, Garside groups, Interval groups.}

\begin{small}
	\tableofcontents 
\end{small}

\section{Introduction}\label{SectionIntro}

This article is a follow-up to the paper \cite{BaumNeaRees} 
by the first, third and fourth authors.
We considered in Part I \cite{BaumNeaRees} the unique infinite family of finite Coxeter groups where proper quasi-Coxeter elements exist (it is the family of type $D_n$). In this paper, we provide a complete description of the interval groups for all the quasi-Coxeter elements for all the types of finite Coxeter groups.

Let $(W,R)$ be a finite Coxeter system, and $w$ a quasi-Coxeter element
(defined in Definition~\ref{DefQuasiCoxInGen}).  Every Coxeter element is
a quasi-Coxeter element. Those elements that are not Coxeter elements are called proper quasi-Coxeter elements.
A proper quasi-Coxeter element exists precisely in types $D_n$ for  $n \geq 4$, $E_6$, $E_7$, $E_8$, $F_4$, $H_3$, and $H_4$ (see \cite{BaumGobet}).

In \cite{Carter}, Carter defines a diagram $\Delta$ associated with each conjugacy class of elements of the simply laced Coxeter groups. For 
quasi-Coxeter elements, these diagrams are the Coxeter diagrams and the diagrams shown in Figures~\ref{FigureCarterDiagramDn} to \ref{FigureCarterDiagramF4a1}. We determine the quasi-Coxeter elements for the non-simply laced types $H_3$ and $H_4$ using the computer algebra system 
GAP \cite{GAP4}.

We denote by $G([1,w])$ the interval group related to $w$ (defined in
Definition~\ref{DefIntervalGrpsCoxeter}). 
We describe presentations of the interval groups for the simply laced types by generators and relations all in accordance with the corresponding Carter diagram along with what we call twisted cycle commutator relators, which are defined again from the Carter diagram. We formulate these presentations in the next theorem.

\begin{customthm}{A}\label{ThmPres}
Let $W$ be a simply laced Coxeter group, 
$w$ a proper quasi-Coxeter element of $W$, and $\Delta$ the Carter diagram associated with $w$.
Then the interval group $G([1,w])$ 
	is the quotient of the Artin group $A(\Delta)$
of the Carter diagram $\Delta$ associated with $w$ by the normal 
closure of a set of twisted cycle commutators
		$\tc{\ss_1}{\ss_2}{\ss_3}{\ss_4}$,
one for each 4-cycle $(s_1,s_2,s_3,s_4)$ within $\Delta$, where
		$\tc{\ss_1}{\ss_2}{\ss_3}{\ss_4}$ is defined to be $[\ss_1,\ss_2^{-1}\ss_3\ss_4\ss_3^{-1}\ss_2]$.
\end{customthm}

For type $D_n$, Theorem~\ref{ThmPres} is proven in \cite{BaumNeaRees}. We restate that result in this paper as Theorem~\ref{ThmPresDn}. For the exceptional cases $E_6$, $E_7$, and $E_8$, we explain in Section~\ref{SecPres} how we prove this result computationally; that result is stated as Theorem~\ref{ThmEnIntervalProper}.

Theorem~\ref{ThmPres} evokes a similar result for Artin groups of the same types, 
proved in \cite{GrantMarsh,Haley}, 
relating to most (but not all) of the Carter diagrams referred to in Theorem~\ref{ThmPres}. It is proven in this article as 
Theorem~\ref{ThmEnArtin}.

Nice presentations for the  non-simply laced finite Coxeter groups of types $F_4$ and $H_3$, are described in Theorems~\ref{ThmH3_a1}~to~\ref{ThmF4IntervalProper}.

Other important results in Section~\ref{SecPres} concern the poset related to the interval $[1,w]$.  We embed this poset into the poset of subspaces of the ambient 
vector space of the Tits' representation of $(W,R)$ as well as in the poset of 
parabolic subgroups of $W$. Further, we classify in Theorem~\ref{ThmLattice} the cases where $[1,w]$ is a lattice. In particular, the interval group related to the quasi-Coxeter element $H_3(a_2)$ considered in Theorem~\ref{ThmH3_a2}  is a Garside group, and by Theorem~\ref{ThmNonIsom} below, that Garside group is not isomorphic to the Artin group of type $H_3$.

The main result of Section~\ref{SecNonIsom} is the following. 

\begin{customthm}{B}\label{ThmNonIsom}
	For $W$ of type $D_n$ with $n$ even and for all the exceptional types, the interval group $G([1,w])$ of a proper quasi-Coxeter element $w$ is not isomorphic to the Artin group of the same type as $W$.
\end{customthm}

Its proof employs an  adaptation of Tits' methods, which   were originally introduced by Tits in \cite{Tits}, where the study of Artin groups of finite Coxeter groups was initiated. Theorem~\ref{ThmNonIsom} will be a consequence of  Theorems~\ref{NonIsoDn}~and~\ref{NonIsoEn}. A complete proof for type $D_n$ for any $n$ based on different methods is provided by the authors in~\cite{BHNR_Part3}.\\

\noindent
\textbf{Acknowledgements}. The third author would like to thank the DFG as he is funded through the DFG grants BA2200/5-1 and RO1072/19-1. The first and third authors would like to thank Theo Douvropoulos and Thomas Gobet for the fruitful discussions during their visits to Bielefeld University.

\section{Presentations of interval groups for quasi-Coxeter elements}\label{SecPres}

\subsection{Dual approach to Coxeter groups}\label{SubDualApproach}

Let $(W,R)$ be a Coxeter system, and let $T = \bigcup\limits_{w \in W}^{} w^{-1}Rw$ be the set of reflections of $W$. The dual approach to the Coxeter group $W$ is the study of $W$ as a group generated by $T$. Note that the classical approach uses the Coxeter system $(W,R)$ with generating set $R$. 

Each $w \in W$ is a product of reflections in $\Refl$. We define 
$$
\ell_\Refl(w):= \min \{ k \in \mathbb{Z}_{\geq 0} \mid w=\refl_{1}\refl_2 \cdots \refl_{k};\ \refl_i \in \Refl \}
$$
called the reflection length of $w$. Let $w = \refl_1\refl_2 \cdots \refl_k$ with $\refl_i \in \Refl$ and $k=\ell_\Refl (w)$. We call $(\refl_1,\refl_2, \dotsc, \refl_k)$ (or $\refl_1\refl_2 \cdots \refl_k$ by abuse of notation) a reduced decomposition of $w$.

Since the set $\Refl$ of reflections is closed under conjugation, there is a natural way to obtain new reflection decompositions from a given one. The braid group $\mathcal{B}_n$ acts on the set $\Refl^n$ of $n$-tuples of reflections via
\begin{align*}
	\boldsymbol{r}_i (\refl_1 ,\dotsc , \refl_n ) &:= (\refl_1 ,\dotsc , \refl_{i-1} , \hspace*{5pt} \refl_i \refl_{i+1} \refl_i,
	\hspace*{5pt} \phantom{\refl_{i+1}}\refl_i\phantom{\refl_{i+1}}, \hspace*{5pt} \refl_{i+2} ,
	\dotsc , \refl_n), \\
	\boldsymbol{r}_i^{-1} (\refl_1 ,\dotsc , \refl_n ) &:= (\refl_1 ,\dotsc , \refl_{i-1} , \hspace*{5pt} \phantom
	{\refl_i}\refl_{i+1}\phantom{\refl_i}, \hspace*{5pt} \refl_{i+1}\refl_i\refl_{i+1}, \hspace*{5pt} \refl_{i+2} ,
	\dotsc , \refl_n).
\end{align*}
\noindent We call this action of $\mathcal{B}_n$ on $\Refl^n$ the Hurwitz action. It is readily observed that this action restricts to the set of all reduced reflection decompositions of a given element $w \in W$. If the latter action is transitive, then we say that the dual Matsumoto property holds for $w$.\\

\subsection{Quasi-Coxeter elements}

Recall that a Coxeter element $c \in W$ is defined to be any conjugate of the product of all elements of $R$ in some order. A more general notion of (parabolic) quasi-Coxeter elements is described in the next definition. It is borrowed from \cite{BaumGobet}.

\begin{definition}\label{DefQuasiCoxInGen}
	\begin{itemize}
	\item[(a)] A subgroup $P$ of $W$ is called a parabolic subgroup if 
	 there is a simple system $R'= \{ r_1' , \ldots , r_n'\}$ of $W$  such that $P = \langle r_1' , \ldots , r_m' \rangle$ for some $m \leq n$. 
	 \item[(b)] An element $w \in W$ is called   a parabolic quasi-Coxeter element for $(W,T)$ 
if there is  a reduced decomposition $w = t_1  \cdots  t_m$ such that $\langle t_1 , \ldots , t_m \rangle$ is a parabolic subgroup of $W$.
If  this parabolic subgroup is $W$, then we simply    call $w$  quasi-Coxeter element.
\end{itemize}
\end{definition}

\begin{remark}
Usually, a  parabolic subgroup of $W$ is defined to be the conjugate of a subgroup generated 
by a subset of $R$. 
If $W$ is finite, then Definition~\ref{DefQuasiCoxInGen}(a) coincides with the usual definition of a parabolic subgroup agrees  (see \cite[Section 4]{BaumGobet}).
\end{remark}

Every Coxeter element is a quasi-Coxeter element, and a quasi-Coxeter element is called proper if it is not a Coxeter element.

The dual Matsumoto property characterises the parabolic quasi-Coxeter elements (see Theorem~1.1 in \cite{BaumGobet}).

\begin{theorem}\label{LemmaTransitivityQCox}
	An element  $ w \in W$ is a parabolic quasi-Coxeter element if and only if the dual Matsumoto property holds for $w$.
\end{theorem}

\subsection{Carter diagrams}\label{SubCarterDiagrams}

Let $W$ be a crystallographic Coxeter group (Weyl group),  that is of type $A_n$, $B_n$, $D_n$, $E_6$, $E_7$, or $E_8$ or of type $F_4$.
Let $w$ be an element of a Coxeter group $W$. By Carter \cite{Carter}, there exists a bipartite decomposition of $w$ over the set $T$ of reflections of the form $$w = w_1 w_2 = \underset{w_1}{\underbrace{t_1 t_2 \cdots t_k}}\ \underset{w_2}{\underbrace{t_{k+1} \cdots t_{k+h}}},$$ where $\ell_T(w) = k+h$ and bipartite means that any $t_i$ and $t_j$ ($i \neq j$) in the decomposition of $w_1$ and in
the decomposition of $w_2$ commute.

A Carter diagram $\Delta$ related to this bipartite decomposition of $w$ has vertices that correspond to the elements $t_i$ that appear in the decomposition of $w$, and two vertices $t_i$ and $t_j$ ($i \neq j$) are related by $o(t_it_j)-2$ edges. The Carter diagram is called admissible if each of its cycles contains an even number of vertices. Carter introduced these diagrams in order to classify the conjugacy classes in Weyl groups.

Carter diagrams on $n$ vertices, where $n$ is the cardinality of $R$, describe the conjugacy classes of quasi-Coxeter elements in $W$. Now we will describe a Carter diagram related to each conjugacy class of proper quasi-Coxeter elements that contains a chordless cycle of four vertices. Note that for a Coxeter element, the corresponding Carter diagram is the Coxeter diagram (which has no cycle).

\subsubsection*{Carter diagrams in type $D_n$}\label{SubsubCarterDn}

There are $\lfloor n/2 \rfloor$ conjugacy classes of quasi-Coxeter elements in type $D_n$ for $n \geq 4$. The following Carter diagram describes these conjugacy classes, where $1 \leq m \leq \lfloor n/2 \rfloor$.

\tikzstyle{every node}=[circle, draw, fill=black!50,
inner sep=0pt, minimum width=4pt]

\begin{small}
\begin{figure}[H]
	\begin{tikzpicture}
		
		\node[draw, shape=circle, label=below:$s_2$] (2) at (0,0) {};
		\node[draw, shape=circle, label=below:$s_3$] (3) at (1,0) {};
		\node[] (0) at (2,0) {};
		\node[] (00) at (3,0) {};
		\node[draw, shape=circle, label=below:$s_{m-1}$] (4) at (4,0) {};
		\node[draw, shape=circle, label=above:$s_m$] (5) at (5,0) {};
		
		\node[draw, shape=circle, label=above:$s_{m+1}$] (6) at (6,1) {};
		\node[draw, shape=circle, label=below:$s_1$] (7) at (6,-1) {};
		\node[draw, shape=circle,label=above:$s_{m+2}$] (8) at (7,0) {};
		\node[draw, shape=circle,label=below:$s_{m+3}$] (9) at (8,0) {};
		\node[] (10) at (9,0) {};
		\node[] (11) at (10,0) {};
		\node[draw,shape=circle, label=below:$s_{n-1}$] (12) at (11,0) {};
		\node[draw, shape=circle, label=below:$s_n$] (13) at (12.2,0) {};
		
		\draw[-] (2) to (3);
		\draw[-] (3) to (0);
		\draw[dashed,-] (0) to (00);
		\draw[-] (00) to (4);
		\draw[-] (4) to (5);
		\draw[-] (5) to (6);
		\draw[-] (5) to (7);
		\draw[-] (8) to (6);
		\draw[-] (8) to (7);
		\draw[-] (8) to (9);
		\draw[-] (9) to (10);
		\draw[dashed,-] (10) to (11);
		\draw[-] (11) to (12);
		\draw[-] (12) to (13);
		
	\end{tikzpicture}\caption{Carter diagram $\Delta_{m,n}$ of type $D_n$.}\label{FigureCarterDiagramDn}
\end{figure}
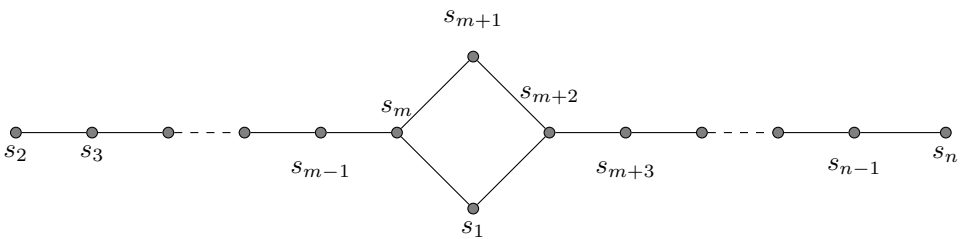
\end{small}

Now we discuss the exceptional cases. For the related Carter diagrams, we will use the notation of Carter (see \cite{Carter}).

\subsubsection*{Carter diagrams in type $E_6$}\label{SubsubCarterE6}

There are two conjugacy classes of proper quasi-Coxeter elements, whose Carter diagrams are illustrated in Figure~\ref{FigureCarterDiagramsE6}.

\tikzstyle{every node}=[circle, draw, fill=black!50,
inner sep=0pt, minimum width=4pt]

\begin{small}
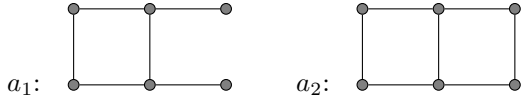
\begin{figure}[H]
	\begin{center}
		\begin{tabular}{lcl}
			$a_1$:\quad \begin{tikzpicture}
				
				\node[] (00) at (0,0) {};
				\node[] (10) at (1,0) {};
				\node[] (20) at (2,0) {};
				\node[] (01) at (0,1) {};
				\node[] (11) at (1,1) {};
				\node[] (21) at (2,1) {};
				
				\draw[-] (00) to (10);
				\draw[-] (10) to (20);
				\draw[-] (01) to (11);
				\draw[-] (11) to (21);
				\draw[-] (00) to (01);
				\draw[-] (10) to (11);
				
			\end{tikzpicture}
			&\quad&
			$a_2$:\quad \begin{tikzpicture}
				
				\node[draw,shape=circle] (00) at (0,0) {};
				\node[draw,shape=circle] (10) at (1,0) {};
				\node[draw,shape=circle] (20) at (2,0) {};
				\node[draw,shape=circle] (01) at (0,1) {};
				\node[draw,shape=circle] (11) at (1,1) {};
				\node[draw,shape=circle] (21) at (2,1) {};
				
				\draw[-] (00) to (10);
				\draw[-] (10) to (20);
				\draw[-] (01) to (11);
				\draw[-] (11) to (21);
				\draw[-] (00) to (01);
				\draw[-] (10) to (11);
				\draw[-] (20) to (21);
				
			\end{tikzpicture}
		\end{tabular}
	\end{center}\caption{Carter diagrams $E_6(a_i)$ for $i=1,2$.}\label{FigureCarterDiagramsE6}
\end{figure}
\end{small}

\subsubsection*{Carter diagrams in type $E_7$}\label{SubsubCarterE7}

There are four conjugacy classes of proper quasi-Coxeter elements, whose Carter diagrams are illustrated in Figure~\ref{FigureCarterDiagramsE7}.

\begin{small}
\begin{figure}[H]
	\begin{center}
		\begin{tabular}{lcl}
			$a_1$:\quad		\begin{tikzpicture}
				
				\node[draw,shape=circle] (00) at (0,0) {};
				\node[draw,shape=circle] (10) at (1,0) {};
				\node[draw,shape=circle] (20) at (2,0) {};
				\node[draw,shape=circle] (01) at (0,1) {};
				\node[draw,shape=circle] (11) at (1,1) {};
				\node[draw,shape=circle] (21) at (2,1) {};
				\node[draw,shape=circle] (31) at (3,1) {};
				
				\draw[-] (00) to (10);
				\draw[-] (10) to (20);
				\draw[-] (01) to (11);
				\draw[-] (11) to (21);
				\draw[-] (00) to (01);
				\draw[-] (10) to (11);
				\draw[-] (31) to (21);
				
			\end{tikzpicture}
			&\quad&
			$a_2$:\quad	\begin{tikzpicture}
				
				\node[draw,shape=circle] (00) at (0,0) {};
				\node[draw,shape=circle] (10) at (1,0) {};
				\node[draw,shape=circle] (20) at (2,0) {};
				\node[draw,shape=circle] (01) at (0,1) {};
				\node[draw,shape=circle] (11) at (1,1) {};
				\node[draw,shape=circle] (21) at (2,1) {};
				\node[draw,shape=circle] (m11) at (-1,1) {};
				
				\draw[-] (00) to (10);
				\draw[-] (10) to (20);
				\draw[-] (01) to (11);
				\draw[-] (11) to (21);
				\draw[-] (00) to (01);
				\draw[-] (10) to (11);
				\draw[-] (m11) to (01);
				
			\end{tikzpicture}
			\\
			&\quad&\\
			$a_3$:\quad \begin{tikzpicture}
				
				\node[draw,shape=circle] (00) at (0,0) {};
				\node[draw,shape=circle] (10) at (1,0) {};
				\node[draw,shape=circle] (20) at (2,0) {};
				\node[draw,shape=circle] (01) at (0,1) {};
				\node[draw,shape=circle] (11) at (1,1) {};
				\node[draw,shape=circle] (21) at (2,1) {};
				\node[draw,shape=circle] (31) at (3,1) {};
				
				\draw[-] (00) to (10);
				\draw[-] (10) to (20);
				\draw[-] (01) to (11);
				\draw[-] (11) to (21);
				\draw[-] (00) to (01);
				\draw[-] (10) to (11);
				\draw[-] (20) to (21);
				\draw[-] (31) to (21);
				
			\end{tikzpicture}
			&\quad&
			$a_4$:\quad		\begin{tikzpicture}
				
				\node[draw,shape=circle] (00) at (0,0) {};
				\node[draw,shape=circle] (10) at (1,0) {};
				\node[draw,shape=circle] (20) at (2,0) {};
				\node[draw,shape=circle] (01) at (0,1) {};
				\node[draw,shape=circle] (11) at (1,1) {};
				\node[draw,shape=circle] (21) at (2,1) {};
				\node[draw,shape=circle] (12) at (1,2) {};
				
				\draw[-] (00) to (10);
				\draw[-] (10) to (20);
				\draw[-] (01) to (11);
				\draw[-] (11) to (21);
				\draw[-] (00) to (01);
				\draw[-] (10) to (11);
				\draw[-] (20) to (21);
				\draw[-] (12) to (21);
				\draw[-] (12) to (01);
				
			\end{tikzpicture}
		\end{tabular}
	\end{center}\caption{Carter diagrams $E_7(a_i)$ for $i=1,\ldots, 4$}\label{FigureCarterDiagramsE7}
\end{figure}
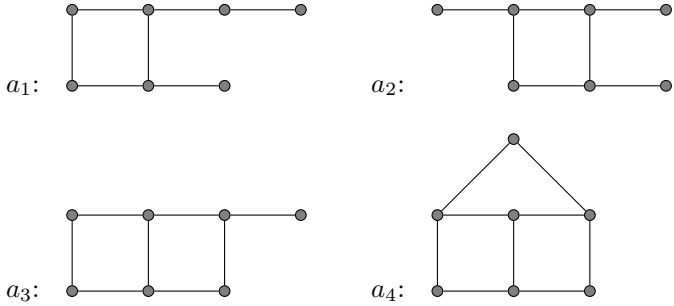
\end{small}

\subsubsection*{Carter diagrams in type $E_8$}

There are eight conjugacy classes of proper quasi-Coxeter elements, whose Carter diagrams are illustrated in Figure~\ref{FigureCarterDiagramsE8}.

\begin{small}
\begin{figure}[H]
	\begin{center}
		\begin{tabular}{clc}
			$a_1$:\quad		\begin{tikzpicture}
				
				\node[draw,shape=circle] (00) at (0,0) {};
				\node[draw,shape=circle] (10) at (1,0) {};
				\node[draw,shape=circle] (20) at (2,0) {};
				\node[draw,shape=circle] (01) at (0,1) {};
				\node[draw,shape=circle] (11) at (1,1) {};
				\node[draw,shape=circle] (21) at (2,1) {};
				\node[draw,shape=circle] (31) at (3,1) {};
				\node[draw,shape=circle] (41) at (4,1) {};
				
				\draw[-] (00) to (10);
				\draw[-] (10) to (20);
				\draw[-] (01) to (11);
				\draw[-] (11) to (21);
				\draw[-] (00) to (01);
				\draw[-] (10) to (11);
				\draw[-] (31) to (21);
				\draw[-] (41) to (31);
				
			\end{tikzpicture}
			&\quad&
			$a_2$:\quad \begin{tikzpicture}
				
				\node[draw,shape=circle] (00) at (0,0) {};
				\node[draw,shape=circle] (10) at (1,0) {};
				\node[draw,shape=circle] (20) at (2,0) {};
				\node[draw,shape=circle] (01) at (0,1) {};
				\node[draw,shape=circle] (11) at (1,1) {};
				\node[draw,shape=circle] (21) at (2,1) {};
				\node[draw,shape=circle] (31) at (3,1) {};
				\node[draw,shape=circle] (m10) at (-1,0) {};
				
				\draw[-] (00) to (10);
				\draw[-] (10) to (20);
				\draw[-] (01) to (11);
				\draw[-] (11) to (21);
				\draw[-] (00) to (01);
				\draw[-] (10) to (11);
				\draw[-] (31) to (21);
				\draw[-] (m10) to (00);
				
			\end{tikzpicture}
			\\
			&\quad &\\
			$a_3$:\quad	\begin{tikzpicture}
				
				\node[draw,shape=circle] (00) at (0,0) {};
				\node[draw,shape=circle] (10) at (1,0) {};
				\node[draw,shape=circle] (20) at (2,0) {};
				\node[draw,shape=circle] (01) at (0,1) {};
				\node[draw,shape=circle] (11) at (1,1) {};
				\node[draw,shape=circle] (21) at (2,1) {};
				\node[draw,shape=circle] (m11) at (-1,1) {};
				\node[draw,shape=circle] (m10) at (-1,0) {};
				
				\draw[-] (00) to (10);
				\draw[-] (10) to (20);
				\draw[-] (01) to (11);
				\draw[-] (11) to (21);
				\draw[-] (00) to (01);
				\draw[-] (10) to (11);
				\draw[-] (m11) to (01);
				\draw[-] (m10) to (00);
				
			\end{tikzpicture}
			&\quad&
			$a_4$:\quad \begin{tikzpicture}
				
				\node[draw,shape=circle] (00) at (0,0) {};
				\node[draw,shape=circle] (10) at (1,0) {};
				\node[draw,shape=circle] (20) at (2,0) {};
				\node[draw,shape=circle] (01) at (0,1) {};
				\node[draw,shape=circle] (11) at (1,1) {};
				\node[draw,shape=circle] (21) at (2,1) {};
				\node[draw,shape=circle] (31) at (3,1) {};
				\node[draw,shape=circle] (41) at (4,1) {};
				
				\draw[-] (00) to (10);
				\draw[-] (10) to (20);
				\draw[-] (01) to (11);
				\draw[-] (11) to (21);
				\draw[-] (00) to (01);
				\draw[-] (10) to (11);
				\draw[-] (20) to (21);
				\draw[-] (31) to (21);
				\draw[-] (31) to (41);
				
			\end{tikzpicture}
			\\
			&\quad&\\
			$a_5$:\quad \begin{tikzpicture}
				
				\node[draw,shape=circle] (00) at (0,0) {};
				\node[draw,shape=circle] (10) at (1,0) {};
				\node[draw,shape=circle] (20) at (2,0) {};
				\node[draw,shape=circle] (01) at (0,1) {};
				\node[draw,shape=circle] (11) at (1,1) {};
				\node[draw,shape=circle] (21) at (2,1) {};
				\node[draw,shape=circle] (31) at (3,1) {};
				\node[draw,shape=circle] (m10) at (-1,0) {};
				
				\draw[-] (00) to (10);
				\draw[-] (10) to (20);
				\draw[-] (01) to (11);
				\draw[-] (11) to (21);
				\draw[-] (00) to (01);
				\draw[-] (10) to (11);
				\draw[-] (20) to (21);
				\draw[-] (31) to (21);
				\draw[-] (00) to (m10);
				
			\end{tikzpicture}
			&\quad&
			$a_6$:\quad		\begin{tikzpicture}
				
				\node[draw,shape=circle] (00) at (0,0) {};
				\node[draw,shape=circle] (10) at (1,0) {};
				\node[draw,shape=circle] (20) at (2,0) {};
				\node[draw,shape=circle] (01) at (0,1) {};
				\node[draw,shape=circle] (11) at (1,1) {};
				\node[draw,shape=circle] (21) at (2,1) {};
				\node[draw,shape=circle] (31) at (3,1) {};
				\node[draw,shape=circle] (30) at (3,0) {};
				
				\draw[-] (00) to (10);
				\draw[-] (10) to (20);
				\draw[-] (01) to (11);
				\draw[-] (11) to (21);
				\draw[-] (00) to (01);
				\draw[-] (10) to (11);
				\draw[-] (20) to (21);
				\draw[-] (31) to (21);
				\draw[-] (31) to (30);
				\draw[-] (20) to (30);
				
			\end{tikzpicture}
			\\
			&\quad&\\
			$a_7$:\quad \begin{tikzpicture}
				
				\node[draw,shape=circle] (00) at (0,0) {};
				\node[draw,shape=circle] (10) at (1,0) {};
				\node[draw,shape=circle] (20) at (2,0) {};
				\node[draw,shape=circle] (01) at (0,1) {};
				\node[draw,shape=circle] (11) at (1,1) {};
				\node[draw,shape=circle] (21) at (2,1) {};
				\node[draw,shape=circle] (12) at (1,2) {};
				\node[draw,shape=circle] (13) at (2,2) {};
				
				\draw[-] (00) to (10);
				\draw[-] (10) to (20);
				\draw[-] (01) to (11);
				\draw[-] (11) to (21);
				\draw[-] (00) to (01);
				\draw[-] (10) to (11);
				\draw[-] (20) to (21);
				\draw[-] (12) to (21);
				\draw[-] (12) to (01);
				\draw[-] (13) to (12);
				
			\end{tikzpicture}
			&\quad&
			$a_8$:\quad \begin{tikzpicture}[scale=.7]
				
				\node[draw,shape=circle] (1) at (0,2) {};
				\node[draw,shape=circle] (2) at (2,2) {};
				\node[draw,shape=circle] (3) at (0,0) {};
				\node[draw,shape=circle] (4) at (2,0) {};
				\node[draw,shape=circle] (5) at (0.8,2.8) {};
				\node[draw,shape=circle] (6) at (2.8,2.8) {};
				\node[draw,shape=circle] (7) at (0.8,1) {};
				\node[draw,shape=circle] (8) at (2.8,1) {};
				
				\draw[-] (1) to (2);
				\draw[-] (1) to (5);
				\draw[-] (1) to (3);
				\draw[-] (2) to (6);
				\draw[-] (2) to (4);
				\draw[-] (3) to (7);
				\draw[-] (3) to (4);
				\draw[-] (4) to (2);
				\draw[-] (4) to (8);
				\draw[-] (5) to (6);
				\draw[-] (5) to (1);
				\draw[-] (5) to (7);	
				\draw[-] (7) to (8);
				\draw[-] (8) to (6);	
				
			\end{tikzpicture}
		\end{tabular}
	\end{center}\caption{Carter diagrams $E_8(a_i)$ for $i=1,\ldots,8$}\label{FigureCarterDiagramsE8}
\end{figure}
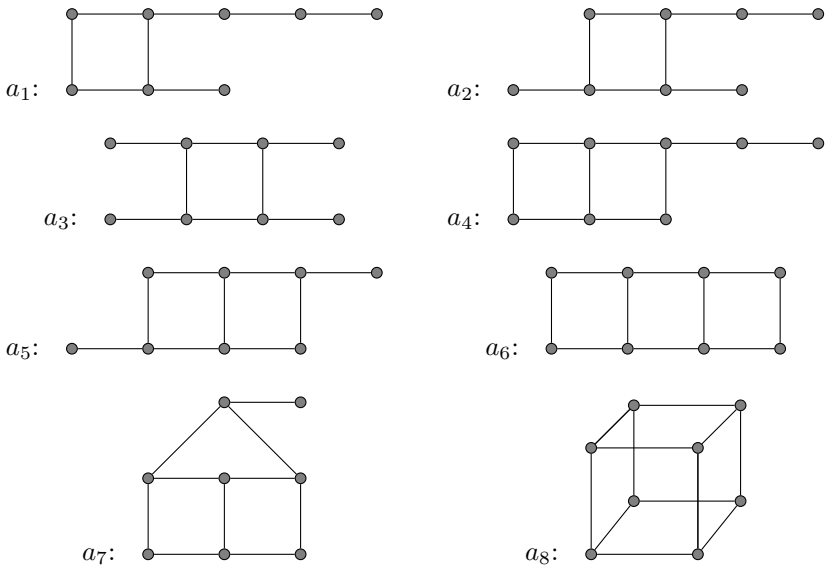
\end{small}

\subsubsection*{Carter diagrams in type $F_4$}\label{SubsubCarterF4}

There is one conjugacy class of proper quasi-Coxeter elements, whose Carter diagram is illustrated in Figure~\ref{FigureCarterDiagramF4a1}.

\begin{figure}[H]
	\begin{center}
		\begin{tikzpicture}
			
			\node[draw,shape=circle] (00) at (0,0) {};
			\node[draw,shape=circle] (1,1) at (1,1) {};
			\node[draw,shape=circle] (m11) at (-1,1) {};
			\node[draw,shape=circle] (02) at (0,2) {};
			
			\draw[-] (00) to (m11);
			\draw[-] (11) to (02);
			\draw[-,double] (m11) to (02);
			\draw[-,double] (11) to (00);
			
		\end{tikzpicture} 
	\end{center}\caption{Carter diagram $F_4(a_1)$.}\label{FigureCarterDiagramF4a1}
\end{figure}
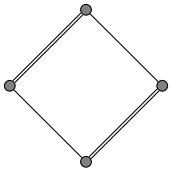

\subsection{Non-crossing partitions for quasi-Coxeter elements}\label{SubLattice}
Let $V$ be the ambient space of the Tits representation  $W \leq GL(V)$  (see \cite[Chapter V, 4.3]{Bourbaki}). Further, let $w \in W$ be a quasi-Coxeter element in $W$.
In this section we define the set of non-crossing partitions $[1,w]$,
and embed it into the set of parabolic subgroups of $W$ as well as into the set of subspaces of $V$. The two sets naturally both carry the structure of posets. We will see that each pair is either isomorphic or anti-isomorphic.

We also analyse the lattice property for the poset of non-crossing partitions for $w$
for the finite Coxeter systems and all the  quasi-Coxeter elements.

\subsubsection*{The poset $([1,w], \preceq)$}

We start by defining left and right division in $W$.

\begin{definition}\label{DefRelation}
	
	We say that $x \in W$ is a left divisor of $w$, and write $x \preceq w$, if $w = xy$ 
	and  $\ell_{\Refl}(w) = \ell_{\Refl}(v) + \ell_{\Refl}(u)$ for some with $y \in W$.
\end{definition}

Notice that $x \preceq w$  holds if and only if  $\ell_{\Refl}(w) = \ell_{\Refl}(x) + \ell_{\Refl}(x^{-1}w)$, which holds if and only if every reduced $T$-decomposition of $x$ can be extended to a reduced $T$-decomposition of $w$.

\begin{definition}\label{DefAbsoluteOrder}
The relation $\preceq$ is an order relation and is called the  absolute order on $W$.
	The interval $[1,w]$ related to an element $w \in W$ is defined to be the set of left divisors of $w$ for $\preceq$.
	We also call $[1,w]$ the set of non-crossing partitions for $w$.
	
	We define division from the right similarly. We say that $y$ is a right divisor of $w$, and write $y \preceq_{r} w$, if  there is $x$ in $W$ such that $w = xy$
	 and $\ell_{\Refl}(w) = \ell_{\Refl}(v) + \ell_{\Refl}(u)$. We also define the interval $[1,w]_{r}$ of right divisors of an element $w \in W$.
\end{definition}
	
It follows that the pair $([1,w], \preceq)$ is a poset. In Theorem~\ref{ThmLattice}, we will show that, apart from in 
type $H_3$ and $H_4$, it is a lattice if and only if the element $w$ is a Coxeter element. In type $H_3$, the poset $[1,w]$ is a lattice for each quasi-Coxeter element $w \in W$.\\
	
We recall Corollary~6.11 of \cite{BaumGobet}, which we state here as a lemma.
	
\begin{lemma}\label{PrefixQuasiCox}
	Let $w \in W$ be a quasi-Coxeter element. Then every element in $[1,w]$ is a parabolic quasi-Coxeter element.
\end{lemma}
	
\subsubsection*{The poset $({\cal P}(w), \leq)$}

Let $x$ be an element in the interval $[1,w]$. According to Lemma~\ref{PrefixQuasiCox}, $x$ is a parabolic quasi-Coxeter element. Therefore, there exists a reduced $T$-decomposition $x = t_1 \cdots t_k$ such that $P_x:= \langle t_1, \ldots , t_k \rangle$ is a parabolic subgroup of $W$, where $t_1, \cdots t_k \in T$. Recall that the parabolic closure of $x$ is the intersection of all parabolic subgroups that contain $x$. It is again a parabolic subgroup.

\begin{lemma}\label{PropParaSbgrp}
The following  properties hold.
\begin{itemize}
\item[(a)] $P_x = \langle t_1, \dots, t_k \rangle$ if $x = t_1 \cdots t_k$ is a  reduced $T$-decomposition with  $t_1, \ldots, t_k \in T$.
\item[(b)] $P_x = \langle t\in T \mid t \preceq x\rangle$.
\item[(c)] $P_x = \langle y \in W \mid y \preceq x\rangle$.
\item[(d)] $P_x$ is the parabolic closure of $x$.
\item[(e)]   $P_x = \langle w \in W \mid  w(v) = v$ for all $v \in \fix(x) \rangle$ where $\fix(x) = \{v \in V \mid x(v) = v\}$.
\end{itemize}
\end{lemma}

\begin{proof} Statement (a) is a consequence of \cite[Proposition~4.3]{BaumGobet} and Theorem~\ref{LemmaTransitivityQCox}.  
	By the definition $P_x$ is a subgroup of  $P :=  \langle t\in T \mid t \preceq x\rangle$. If $t \preceq x$ for some $t \in T$, then there is a reduced $T$-decomposition of $x$ including $t$. Therefore, $t$ is in $P_x$ by (a), which shows 
	$P \subseteq P_x$, and $P_x = P$. This shows (b). The proof of (c) is done analogously.  Assertions  (d) and (e) follow from   \cite[Theorem~1.4]{BDSW} and  \cite[Section~4]{BaumGobet}, respectively (for (e) see also
	\cite[Chapter V, 1.6]{Bourbaki}).
\end{proof}

Thus, the definition of $P_x$ is independent of the chosen reduced $T$-decomposition of $x$.
We set ${\cal P}(w):= \{ P_x~|~ x \in [1,w]\}$.  Then the  subgroup relation $\leq$    defines an order relation on ${\cal P}(w)$, and  $({\cal P}(w), \leq)$ is a poset.

\subsubsection*{The poset $({\cal F}(w),\subseteq)$}

For each $x \in [1,w]$, consider the subspace $\Fix(x):= \ker(x- id) =  \{v \in V~|~ x(v) = v\}$ of $V$  (see Lemma~\ref{PropParaSbgrp}(e)). Consider the commutator subgroup $\Mov(x):= [V,x] =  \im(x - id)$ in the semi-direct product
$V \rtimes \langle x \rangle$. Then, due to Maschke's theorem $V = \Fix(x) \oplus \Mov(x)$.  It is an easy calculation to see that $\Fix(x)$ and  $\Mov(x)$ are perpendicular  and therefore
$\Fix(x) = \Mov(x)^\perp$.  More generally, we  consider for $X \subseteq W$ the subspace $\Fix(X):=   \{v \in V~|~ x(v) = v ~\mbox{for all}~x \in X\}$  of $V$. Note that for the parabolic subgroup $P_x$, we have $\Fix(P_x) = \Fix(x)$ for every $x \in [1,w]$.

Set ${\cal F}(w):= \{\Fix(x)~|~x \in [1,w]\}$. Then  the inclusion $\subseteq$  on sets 
defines an order relation on ${\cal F}(w)$, and 
$({\cal F}(w), \subseteq)$ is a poset. Further, notice that ${\cal F}(w)$ contains all the reflection hyperplanes of $W$,  as $T$ is contained in $[1,w]$ by  \cite[Lemma~1.2.1 (i)]{BessisDualMonoid}.
By 	\cite[Chapter V,  1.6]{Bourbaki} every subspace $F$ of  ${\cal F}(w)$ is the intersection of some of the reflection hyperplanes. 
Also note that ${\cal F}(w)$  does not necessarily  contain the intersection of any two of its elements; see Corollary~\ref{IntersectionSubspace}.

\subsubsection*{Isomorphisms and anti-isomorphisms between these posets}

Next, we show that the posets $([1,w], \preceq)$ and  $({\cal F}(w), \subseteq )$ are anti-isomorphic.
Brady and Watt  observed that if $x \preceq z$  for some $x,z \in W$, then 
the action of $x$ on $V/\Fix(x)$ is determined by the action of $z$ on $\Mov(z)$ 
and concluded from this that if $a \in O(V)$ and $U$ a subspace of $\Mov(a)$,
then there is a unique $b \in O(V)$ such that $b \preceq a$ and $\Mov(b) = U$
(see \cite[Theorem~1]{BW02}). From this, we derive the following anti-isomorphism.

\begin{proposition}\label{AntiIsoIntervallFix}
	The map $ \fix:  [1,w]  \rightarrow    {\cal F}(w), x \mapsto \Fix(x)$
	is an anti-isomorphism between the posets $([1,w], \preceq)$ and $ ({\cal F}(w), \subseteq)$.
\end{proposition}
\begin{proof}
	If $\fix(x) = \fix(y)$ for two elements $x,y$ in $[1,w]$, then $\Mov(x) = \fix(x)^\perp = \fix(y)^\perp = \Mov(y)$. Therefore, we obtain with \cite[Theorem~1]{BW02} that $ \fix$ is injective.
	By the definition  of  $ {\cal F}(w)$, it is surjective as well.
	
	Next we show that $\fix$ is an anti-isomorphism between the two posets.
	If $x,y \in [1,w]$ such that $x \preceq y$, then $\fix(y) \subseteq \fix(x)$.
	So, assume that  $\fix(y) \subseteq \fix(x)$ for some $x,y \in [1,w]$.  Then $\Mov(x) \subseteq \Mov(y)$.
	According to  \cite[Theorem~1]{BW02}, there is a unique $z \in O(V)$ with $z \preceq y$  such that
	$\Mov(z) = \Mov(x)$. We conclude from the transitivity of the relation $\preceq$ that $x,y $ are in $[1,w]$.
	Then it is a consequence of \cite[Theorem~1]{BW02} that $x = z$, and $x \preceq y$.
	Thus, $\fix$ is an anti-isomorphism between the two posets.
\end{proof}

From Lemma~\ref{PropParaSbgrp}(e), we immediately derive the following.

\begin{lemma}\label{Fixx}
 We have $\fix(P_x) = \fix(x)$ for all  $x \in [1,w]$.
\end{lemma}

Next we consider the posets $[1,w]$ and ${\cal P}(w)$.

\begin{proposition}\label{IsoIntervalParabolic}
	The map $p:  [1,w]  \rightarrow    {\cal P}(w), x \mapsto P_x$
	is an isomorphism between the posets $([1,w], \preceq)$ and $ ({\cal P}(w), \leq)$.
\end{proposition}
\begin{proof}
By Lemma~\ref{PropParaSbgrp}, the map $p$ is well-defined.
We first show injectivity. Let $x, y \in [1,w]$ such that $P_x = P_y$. Then $\fix(P_x) = \fix(P_y)$, and Lemma~\ref{Fixx} and Proposition~\ref{AntiIsoIntervallFix} yield $x = y$. 
By definition of $ {\cal P}(w)$, the map $p$ is surjective as well.

It remains to show that $p$ is an isomorphism of posets.
Let $x,y \in [1,w]$ such that $x \preceq y$. Then, by definition of 
$p$, we get $P_x \leq P_y$.
Let $P_x \leq P_y$ for some $x,y \in [x,y]$. Then $\fix(P_y) \leq \fix(P_x)$. We conclude $x \preceq y$ by applying Lemma~\ref{Fixx} and  Proposition~\ref{AntiIsoIntervallFix}.
\end{proof}

Combining Propositions \ref{AntiIsoIntervallFix} and \ref{IsoIntervalParabolic}, we immediately get the following.

\begin{cor}\label{Non-isoParabolicFix}
	The posets $ ({\cal P}(w), \leq)$ and  $ ({\cal F}(w), \subseteq)$ are anti-isomorphic.
\end{cor}	

\subsubsection*{The lattice property}
	
\begin{theorem}\label{ThmLattice}

Let $w$ be a quasi-Coxeter element in a finite Coxeter group $W$.

In the simply laced types, the poset $([1,w],\preceq)$ is a lattice if and only if $w$ is a Coxeter element.

In type $H_3$, the poset $([1,w],\preceq)$ is always a lattice.

In type $H_4$, $([1,w],\preceq)$ is a lattice if and only if $w$ is a Coxeter element or a proper quasi-Coxeter element of order $30$, that is the Coxeter number in type $H_4$.

In type $F_4$, $([1,w],\preceq)$ is a lattice if and only if $w$ is a Coxeter element.

\end{theorem}

\begin{proof}

When $w$ is a Coxeter element, the fact that $([1,w],\preceq)$ is a lattice was shown in \cite{BessisDualMonoid} and \cite{BradyWatt}. Then, consider $w$ to be a proper quasi-Coxeter element.

For type $D_n$, we have shown in Proposition~6.6 in \cite{BaumNeaRees} that the poset $([1,w],\preceq)$ is not a lattice, by showing the result in type $D_4$ and then applying Theorem~2.1 of Dyer \cite{Dyer}. 

Consider types $E_6$, $E_7$, and $E_8$. Since the Carter diagram related to each conjugacy class of proper quasi-Coxeter elements contains a $4$-cycle (that is, a type $D_4$ cycle), as illustrated in the figures of Section~\ref{SubCarterDiagrams}, the same Theorem~2.1 of Dyer applies. Hence we also deduce that the posets are not lattices. 

Using GAP \cite{GAP4}, we show the statement of the theorem for types $H_3$, $H_4$, and $F_4$.
\end{proof}

\begin{cor}\label{IntersectionSubspace}
The posets  $ ({\cal P}(w), \leq)$ and $ ({\cal F}(w), \subseteq)$
are not lattices. In particular, there are two subspaces $F_1, F_2$ in ${\cal F}(w)$, whose intersection is not in ${\cal F}(w)$.
\end{cor}
\begin{proof} The first assertions are consequences of the 
	isomorphism and anti-isomorphism results of the given posets. As ${\cal F}(w)$ is not a lattice, there is a bowtie in 
	that poset (see \cite[Proposition~1.5]{BM10}), which yields the second claim.
	\end{proof}

\subsection{Interval groups for quasi-Coxeter elements}\label{SubsectionIntervalGrpsOfQCoxElts}

Let $w$ be a quasi-Coxeter element of $W$. Consider the interval $[1,w] = \{w \in W\ |\ w \preceq w\}$. The interval group related to the interval $[1,w]$ is defined as follows.

\begin{definition}\label{DefIntervalGrpsCoxeter}
	
	We define the group $G([1,w])$ by a presentation with set of generators $\bs{[1,w]}$ in bijection with the interval $[1,w]$, and relations corresponding to the relations in $[1,w]$, meaning that $\bs{uv} = \bs{r}$ if $u,v,r \in [1,w]$, $uv=r$, and $u \preceq r$ i.e. $\ell_{\Refl}(r) = \ell_{\Refl}(u) + \ell_{\Refl}(v)$.
	
\end{definition}

By transitivity of the Hurwitz action on the set of reduced decompositions of $w$ (see Theorem \ref{LemmaTransitivityQCox}), we have the following result.

\begin{proposition}\label{PropDualPresDn}
	Let $w \in W$ be a quasi-Coxeter element, and let $\RRefl \subset \bs{[1,w]}$ be the copy of the set of reflections $\Refl$ in $W$. Then $$G([1,w]) = \langle  \RRefl ~|~  \rrefl \rrefl' = \rrefl' \rrefl'' ~\mbox{for}~\rrefl , \rrefl',  \rrefl''  \in   \RRefl~\mbox{if}~\refl \neq \refl' , \refl'' \in \Refl~\mbox{and}~\refl \refl' = \refl' \refl'' \preceq w \rangle$$
	is a presentation of the interval group with respect to $w$.
\end{proposition}

Notice that the relations presented in Proposition~\ref{PropDualPresDn} are the relations that are visible on the elements of length $2$ in the poset $([1,w],\preceq)$. We call them the dual braid relations (as in \cite{BessisDualMonoid}).

The following result is due to Michel as stated by Bessis in \cite{BessisDualMonoid} (Theorem~0.5.2) and explained on page 318 of Chapter VI in \cite{DehornoyEtAl} (see also \cite{BessisDigneMichel}). It is the main theorem in interval Garside theory.

\begin{theorem}\label{TheoremMainThmIntGarTheory}
	If for $v \in W$, the two intervals $[1,v]$ and $[1,v]_{r}$ are equal (we say that $v$ is balanced) and if the posets $([1,v],\preceq)$ and $([1,v]_{r},\preceq_{r})$ are lattices, then the interval group $G([1,v])$ is an interval Garside group.
\end{theorem}

Since $\Refl$ is stable under conjugation, quasi-Coxeter elements are always balanced. The only obstruction to obtain interval Garside groups is the lattice property. The following is a consequence of Theorem~\ref{ThmLattice}.

\begin{theorem}\label{ThmGarsideArtin}
	Let $c$ be a Coxeter element. Then the  posets $([1,c],\preceq)$ and $([1,c],\preceq_r)$ are lattices; hence the interval group $G([1,c])$ is a Garside group. The group $G([1,c])$ is isomorphic to the Artin group associated with $W$. 
\end{theorem}

Note that Garside groups are desirable since they enjoy important group-theoretical, homological, and homotopical properties. See \cite{DehornoyEtAl} for a treatment on the foundations of Garside theory.

\subsection{Presentations for the interval groups for the quasi-Coxeter elements}\label{PraesentIntervalGrp}

Let $G$ be a group containing elements $\ss_1$, $\ss_2$, $\ss_3$, and $\ss_4$
that satisfy the relations of the Artin group corresponding to
the $4$-cycle that is illustrated in the figure below.

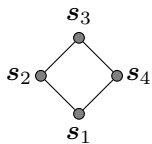
\begin{center}
	\begin{tikzpicture}
		
		\node[draw,shape=circle, label=below:$\ss_1$] (00) at (0,0) {};
		\node[draw,shape=circle, label=right:$\ss_4$] (11) at (0.5,0.5) {};
		\node[draw,shape=circle,label=left:$\ss_2$] (m11) at (-0.5,0.5) {};
		\node[draw,shape=circle,label=above:$\ss_3$] (02) at (0,1) {};
		
		\draw[-] (00) to (m11);
		\draw[-] (11) to (02);
		\draw[-] (m11) to (02);
		\draw[-] (11) to (00);
		
	\end{tikzpicture}
\end{center}

We associate two words with this $4$-cycle, which we call the {\em cycle commutator} and the {\em twisted cycle commutator}, and which we define by
 $$\cc{\ss_1}{\ss_2}{\ss_3}{\ss_4} := [\ss_1,\ss_2\ss_3\ss_4\ss_3^{-1}\ss_2^{-1}],\ \mbox{and}\   
 \tc{\ss_1}{\ss_2}{\ss_3}{\ss_4} := [\ss_1,\ss_2^{-1}\ss_3\ss_4\ss_3^{-1}\ss_2].$$

It is straightforward to check that the four cycle commutators
$$\cc{\ss_1}{\ss_2}{\ss_3}{\ss_4},\, 
\cc{\ss_2}{\ss_3}{\ss_4}{\ss_1},\, 
\cc{\ss_3}{\ss_4}{\ss_1}{\ss_2},\, 
\cc{\ss_4}{\ss_1}{\ss_2}{\ss_3}$$
are equivalent, in the sense that if one of them is a relator of $G$
(i.e. it evaluates to the identity in $G$), then so do the other three,
and the same is true of the corresponding twisted cycle commutators.
It follows from the braid relations between $\ss_2,\ss_3,\ss_4$ that
$\ss_2^{-1}\ss_3\ss_4\ss_3^{-1}\ss_2 =_G \ss_4^{-1}\ss_3\ss_2\ss_3^{-1}\ss_4$.
Hence $\tc{\ss_1}{\ss_2}{\ss_3}{\ss_4}$ and
$\tc{\ss_1}{\ss_4}{\ss_3}{\ss_2}$ are equivalent.
But we cannot deduce the same relationship between
$\cc{\ss_1}{\ss_2}{\ss_3}{\ss_4}$ and
$\cc{\ss_1}{\ss_4}{\ss_3}{\ss_2}$, and so the word 
$\cc{\ss_1}{\ss_2}{\ss_3}{\ss_4}$ must be associated with an oriented 4-cycle of the form $\ss_1 \rightarrow \ss_2 \rightarrow \ss_3 \rightarrow \ss_4 \rightarrow \ss_1$.

Notice also that both the cycle and twisted cycle commutator relators can be written as relations between positive words (see for instance Lemma~5.3 in \cite{BaumNeaRees}).

In the remainder of the paper, we will use the following  abbreviation.
Given a Carter diagram $\Delta$  and $\SS$ a set of generators related 
to the vertices of $\Delta$ we denote by $R(\Delta)$ the set of braid relations
defined by $\Delta$; that is, for each pair of generators $\ss,\tt$, we have the relation
$\ss \tt \ss = \tt \ss \tt$ if the vertices representing $\ss$ and $\tt$ 
are joined by an edge and $\ss \tt = \tt \ss$ if they are not.

\subsubsection*{Type $D_n$}\label{SubsubPresDn}

The presentations in type $D_n$ were the main object of study in our first paper \cite[Theorem A]{BaumNeaRees}.

\begin{theorem}\label{ThmPresDn}
	Let $w$ be a quasi-Coxeter element of the Coxeter group $W$ of type $D_n$ and $\Delta_{m,n}$ its associated Carter diagram,
		as shown in Figure~\ref{FigureCarterDiagramDn}.
	Then the interval group $G_{m,n}:= G([1,w])$ admits a presentation over the generators $\ss_1,\dotsc,\ss_n$ corresponding to the vertices of $\Delta_{m,n}$ together with the relations $R_{m,n}:= R(\Delta_{m,n})$ and the twisted cycle commutator relator
        $\tc{\ss_1} {\ss_{m}}{\ss_{m+1}}{\ss_{m+2}}$, associated with the $4$-cycle $(s_1, s_{m}, s_{m+1}, s_{m+2})$ within $\Delta_{m,n}$.
\end{theorem}

We would also like to draw attention to the alternative presentations for $G_{m,n}$ that are described in \cite{BHNR_Part3}. 

\subsubsection*{Types $E_6$, $E_7$ and $E_8$}

We prove the following results computationally. We will explain later the computational steps used in the proofs. 
Note that the presentations we obtain in Theorem~\ref{ThmEnArtin} were already described in \cite{GrantMarsh} and \cite{Haley}. The main result of this section is Theorem~\ref{ThmEnIntervalProper}.\\

We consider a Carter diagram to be orientable if its edges can be oriented in such a way that each $4$-cycle is oriented. All Carter diagrams of types $E_n(a_i)$ that appear in Figures~\ref{FigureCarterDiagramsE6} to~\ref{FigureCarterDiagramsE8} are orientable except for $E_7(a_4)$, $E_8(a_7)$, and $E_8(a_8)$.

\begin{theorem}\label{ThmEnArtin}
Let $W$ be a Coxeter group of types $E_n$ for $n=6,7$ or $8$. 
	Let $E_n(a_i)$  be an oriented
	 Carter diagram. Then the Artin group $A(E_n)$ associated with $W$ admits a presentation over the generators corresponding to the vertices of $E_n(a_i)$ with the relations $R(E_n(a_i))$ and a set of cycle commutator relators, 
one corresponding to each oriented $4$-cycle in the diagram.
\end{theorem}

Note that given any orientation of the Carter diagram considered in Theorem~\ref{ThmEnArtin} provides the result of the theorem.
Exactly the same diagrams are covered by \cite[Theorem 1.1]{Haley},
which derives a presentation relative to a diagram $\Gamma'$, of an Artin 
group of type $\Gamma$, whenever $\Gamma'$ can be derived from $\Gamma$ by a 
sequence of {\em mutations} (see also \cite{GrantMarsh}). Note that in Theorem~\ref{ThmPres}, Theorem~\ref{ThmEnIntervalProper}, and Theorem~\ref{ThmCameron}, we consider all Carter diagrams of types $E_n(a_i)$ and not only the orientable ones.

\begin{theorem}\label{ThmEnIntervalProper}
Let $W$ be a Coxeter group of types $E_n$ for $n=6,7$ or $8$. 
Let $w$ be a quasi-Coxeter element, and let $\Delta$ be the Carter diagram associated with $w$.

Then the interval group $G([1,w])$ admits a presentation over the generators corresponding to the vertices of $\Delta$ with the relations $R(\Delta)$ and a set of twisted cycle commutator relators, one corresponding to each $4$-cycle in the diagram. 
\end{theorem}

An analogous result for Coxeter groups arises as
a consequence of \cite[Theorem~6.10]{Cameron}, which uses a process called {\em switching}, similar to the process of mutation described in \cite{GrantMarsh, Haley}.

\begin{theorem}\label{ThmCameron}
	Let $W$ be a simply laced Coxeter group, and $\Delta$ be a Carter diagram associated with $W$.
	Then the Coxeter group $W$ admits a presentation over the generators corresponding to the vertices of $\Delta$ with the quadratic relations on the generators, the relations $R(\Delta)$ and a 
	set of cycle commutator relators, one corresponding to each $4$-cycle in $\Delta$. 	
\end{theorem}

Our proofs of 
Theorems~\ref{ThmEnArtin} and \ref{ThmEnIntervalProper} are large computations,
some of them requiring significant computing power over a long period of time.
The presentations were established in sequences of steps, which we describe now. First, we describe the steps that prove Theorem~\ref{ThmEnIntervalProper}.\\

\textit{Step 1.} We choose a representative for each conjugacy class of quasi-Coxeter elements. The computer programs that we used can be found at \\
\url{https://www.math.uni-bielefeld.de/~baumeist/Dual-Coxeter/dual-Coxeter.html}.
Associated to such representative is a Carter diagram (see Section~\ref{SubCarterDiagrams}). We distinguish between the conjugacy classes 
using the orders of the quasi-Coxeter elements. Recall that the order of a Coxeter element is precisely the Coxeter number. For types $E_n$ ($n=6,7,8$), we summarise the orders in the next tables. In each table, the first column contains Carter diagrams and the second column the orders of the corresponding quasi-Coxeter elements.

\begin{center}
\begin{tabular}{|l|l|}
	\hline
	$E_6(a_1)$ & $9$\\
	\hline
	$E_6(a_2)$ & $6$\\
	\hline
	$E_6$ & $12$\\
	\hline
\end{tabular}\quad
\begin{tabular}{|l|l|}
	\hline
	$E_7(a_1)$ & $14$\\
	\hline
	$E_7(a_2)$ & $12$\\
	\hline
	$E_7(a_3)$ & $30$\\
	\hline
	$E_7(a_4)$ & $6$\\
	\hline
	$E_7$ & $18$\\
	\hline
\end{tabular}\quad
\begin{tabular}{|l|l|}
	\hline
	$E_8(a_1)$ & $24$\\
	\hline
	$E_8(a_2)$ & $20$\\
	\hline
	$E_8(a_3)$ & $12$\\
	\hline
	$E_8(a_4)$ & $18$\\
	\hline
	$E_8(a_5)$ & $15$\\
	\hline
	$E_8(a_6)$ & $10$\\
	\hline
	$E_8(a_7)$ & $12$\\
	\hline
	$E_8(a_8)$ & $6$\\
	\hline
	$E_8$ & $30$\\
	\hline
\end{tabular}
\end{center}

\textit{Step 2.} We determine a presentation of the interval group related to the chosen quasi-Coxeter element as follows. First, we determine the length over $T$ of the elements in $W$ and then construct those of length $2$ that divide $w$. From these elements, it is easy to define the dual braid relations that describe our presentation of the interval group.

\textit{Step 3.} We choose a set of reflections $S$ of cardinality the rank of the Coxeter group such that the relations between the corresponding elements in the interval group are those that describe the relations of the Carter diagram related to the conjugacy class of the quasi-Coxeter element. We denote by $\SS$ the copy of $S$ in the interval group.

Using the dual braid relations, we determine an expression over $\SS \cup \SS^{-1}$ of all the generators in $\boldsymbol{T} \backslash \SS$ of the interval group. Finally, we replace the elements that belong to $\boldsymbol{T} \backslash \SS$ in the dual braid relations by their expressions over $\SS \cup \SS^{-1}$.

\textit{Step 4.} Using the package \verb|kbmag| \cite{KBMAG} of GAP \cite{GAP4} and a computation by hand, we show that all the relations other than the one described by the Carter diagram and the corresponding commutator relators simplify in the interval group (see Theorems~\ref{ThmEnArtin},~\ref{ThmEnIntervalProper} for the type of the commutator relators).\\

Now Theorem~\ref{ThmEnArtin} is obtained by considering the conjugacy class of the Coxeter element in \textit{Step 1}, and then applying \textit{Step 2} to \textit{Step 4} for all the related Carter diagrams that appear in Theorem~\ref{ThmEnArtin}. Note that we attempted to construct the presentations for the non-orientable Carter diagrams that are excluded in Theorem~\ref{ThmEnArtin}. But we were unable to complete the computations, hence it seems likely (although it is not proved) that the $E_n$ Artin groups do not have presentations corresponding to those diagrams.\\

As an evidence on how difficult the computation is, consider the case $E_8(a_6)$ of Theorem~\ref{ThmEnIntervalProper} as an example. The related proper quasi-Coxeter element considered in \textit{Step 1} is of order $10$. The number of the dual braid relations we obtain in \textit{Step 2} is $3630$. Theorem~\ref{ThmEnIntervalProper} describes a presentation of the related interval group over $8$ generators and $31$ relations (these are the relations of the Carter diagram $E_8(a_6)$ along with $3$ twisted cycle commutators). The length of the longest relation we simplified in \mbox{\textit{Step 4}} is $2000$.

\subsubsection*{Types $H_3$, $H_4$ and $F_4$}\label{SubPresH3H4}

We start with type $H_4$, where the interval groups of proper quasi-Coxeter elements are sorted out quickly. Actually, there exist ten conjugacy classes of proper quasi-Coxeter elements. None of the intervals are lattices, as we already
mentioned in Section~\ref{SubLattice}. We compare the results of applying the function \verb|LowIndexSubgroupsFpGroup| to these groups within GAP to show that these interval groups are not isomorphic to the Artin group of type $H_4$.\\

For type $H_3$, we have two conjugacy classes of proper quasi-Coxeter elements that we denote by $H_3(a_1)$ and $H_3(a_2)$. Using GAP, we obtain the following results.

\begin{theorem}\label{ThmH3_a1}
	The interval group related to the proper quasi-Coxeter element $H_3(a_1)$ is isomorphic to the Artin group of type $H_3$. Since the interval is a lattice, then this interval group is also a Garside group.
\end{theorem}
	
\begin{theorem}\label{ThmH3_a2}
	The interval group related to $H_3(a_2)$ admits a presentation over three generators $\ss_1,\ss_2,\ss_3$ and the relations are described by the following diagram presentation \begin{center}
		\tikzset{every node/.style={font=\scriptsize}}
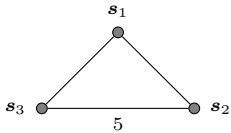
		\begin{tikzpicture}
			\node[circle, draw, fill=black!50,
			inner sep=0pt, minimum width=4pt, label=left:$\ss_3$] (3) at (0,0) {};
			\node[circle, draw, fill=black!50,
			inner sep=0pt, minimum width=4pt, label=right:$\ss_2$] (2) at (2,0) {};
			\node[circle, draw, fill=black!50,
			inner sep=0pt, minimum width=4pt,label=above:$\ss_1$] (1) at (1,1) {};
			
			\draw[-] (1) to (3);
			\draw[-] (1) to (2);
			\draw[-] (2) to node[below] {$5$} (3);
		\end{tikzpicture}
	\end{center}
	along with the two relations
	$$\ss_2\ss_3\ss_2\ss_1\ss_3\ss_2 = \ss_3\ss_2\ss_1\ss_3\ss_2\ss_3, \ (\ss_3\ss_2\ss_1)^3 = (\ss_1\ss_3\ss_2)^3.$$
	Since the interval is a lattice, this interval group is a Garside group.
\end{theorem}

We also show that the interval group of Theorem~\ref{ThmH3_a2} is not isomorphic to the Artin group of type $H_3$ by using \verb|LowIndexSubgroupsFpGroup| within GAP. Hence it defines a new Garside group.

We conjecture that this group is the fundamental group of the complement in $\mathbb{C}^3$ of an algebraic hypersurface.\\

We mention that using the same computational approach that we described previously for the cases $E_6$, $E_7$, and $E_8$, we are able to show the following two results in the case $F_4$.

\begin{theorem}\label{ThmF4Artin}
Let $W$ be a Coxeter group of type $F_4$. Let $F_4(a_1)$ be the Carter diagram illustrated in Figure~\ref{FigureCarterDiagramF4a1} with $(s_1,s_4)$ and $(s_2,s_3)$ the edges with double bonds. Then the Artin group $A(F_4)$ admits a presentation over generators corresponding to the vertices of $F_4(a_1)$ with relations $R(F_4(a_1))$ and the commutator relator $[\ss_2,\ss_3]$.
\end{theorem}

\begin{theorem}\label{ThmF4IntervalProper}
Let $W$ be a Coxeter group of type $F_4$. Let $F_4(a_1)$ be the Carter diagram illustrated in Figure~\ref{FigureCarterDiagramF4a1} with $(s_1,s_4)$ and $(s_2,s_3)$ the edges with double bonds. Then the interval group in this case admits a presentation over generators corresponding to the vertices of $F_4(a_1)$ with relations $R(F_4(a_1))$ and the commutator relators $[\ss_2^{-1}\ss_1\ss_2,\ss_3^{-1}\ss_4\ss_3]$ and $[\ss_2\ss_1\ss_2^{-1},\ss_3\ss_4\ss_3^{-1}]$. Furthermore, it is not isomorphic to the Artin group $A(F_4)$.
\end{theorem}

\section{Non-isomorphism results}\label{SecNonIsom}

In this section,  we show that  the interval groups $G = G([1,w])$ associated with a proper quasi-Coxeter element $w$ of a Coxeter group of 
type $D_n$ with $n$ even, or of type $E_n$ with $n  \in \{ 6,7, 8\}$, are not isomorphic to the respective Artin groups $A$.
This has been already proven for the remaining finite Coxeter groups (types $H_3, H_4, F_4$) in Section~\ref{PraesentIntervalGrp}.

Our first approach was to compare the abelianisations $G/[G,G]$  and $A/[A,A]$ of the respective groups.
Let $\SS$ be the set of generators of the presentations for the interval groups  of
type $D_n$  or $E_6, E_7$ or $E_8$ given in  Section~\ref{PraesentIntervalGrp}.
The related Carter diagrams are connected, and if $\ss,\tt \in \SS$ correspond to neighbours in that diagram, then $\ss \tt \ss = \tt \ss \tt$ and  $\ss \tt \ss^{-1} = \tt^{-1} \ss \tt$ . Therefore, the commutator subgroup 
$G^\prime$ of $G$ contains the elements $\ss \tt \ss^{-1} \tt^{-1} = \tt^{-1} \ss \tt \tt^{-1} =  \tt^{-1} \ss$.
From this it follows that all the elements in $\SS$ are equal in the abelianisation.  Moreover, the
twisted cycle commutator relations still hold  if we identify all the elements in $\SS$.
This  shows that the abelianisations of  the Artin groups as well  as of the interval groups
are isomorphic to $\Z$. (This also shows  that the interval groups are infinite groups.) 

Therefore, we choose the new approach to consider the abelianisations of the pure
Artin and the pure interval groups.
Our strategy is as follows. Let $\Gamma$ be the Coxeter diagram of the Coxeter group $W$. Let $\varphi : A(\Gamma) \longrightarrow W(\Gamma) $ be a homomorphism from the Artin group $A = A(\Gamma)$ to the Coxeter group $W = W(\Gamma)$ such that the kernel $\ker(\varphi)$ is the pure Artin group $\PA(\Gamma)$.  We call this the canonical epimorphism from $A$ to $W$.
Tits proved that
the abelianisation of  $\PA(\Gamma)$ is isomorphic to the free abelian group of rank
$|T|$ (see \cite{Tits}).  

Note that the  epimorphism $\varphi$ sends  the element $\ss \in \SS$ to the respective reflection in  $W$.  
By Theorem~\ref{ThmCameron} there is also such an epimorphism from $G$ to 
$W$, whose kernel is denoted by $K$ and we call it the pure interval group. We show that 
the abelianisation of $K$ is of rank at most $|T|-2$.
Thereby we obtain a contradiction for  the types $D_n$ with $n$ even, or $E_6, E_7, E_8$ by applying the following result of Cohen and Paris \cite{CP}.

\begin{theorem}\label{CohenParis}
Let $A(\Gamma)$ be an Artin group of type $D_n$ with $n$ even, or $E_n$ with 
$n \in \{ 6,7, 8\}$.  Then the canonical epimorphism is the unique epimorphism from 
$A(\Gamma)  $ to $W (\Gamma) $ up to automorphisms of $W(\Gamma)$. 

In the case $D_n$ with $n$ odd, there are three epimorphisms  from 
$A(\Gamma)  $ to $W (\Gamma) $ up to automorphisms of $W(\Gamma)$.
\end{theorem}

We follow the proof of Tits \cite{Tits} in the calculation of the abelianisation of the kernel $K$.
We first sketch his approach, then we discuss the interval groups of type 
$D_n$ in detail and in the last part the interval groups for $E_6,E_7$ and $E_8$.

\subsection{The abelianisation of the pure  Artin group}\label{SubTitsArtin}

Tits used the following  notation in \cite{Tits}. He took a Coxeter system  $(W,R)$  of type $\Gamma$ with simple system 
$R = \{ r_i ~|~1 \leq i \leq n\}$ and set of reflections $T$ (he named it   $S$).
Let $I = \{1, \ldots , n\}$ and denote by ${\bf I}$ the free group on $I$.  We denote by $R(\Gamma)$ the braid relations determined by the 
Coxeter graph $\Gamma$. 

Define a homomorphism from the free group ${\bf I}$ into the Coxeter group $W$  by
\begin{itemize}
\item $r : {\bf I} \rightarrow W$ by $r(i):=r_i$  with kernel   $L := \ker(r)$. 
\end{itemize}
 Then we have
$W = \langle r_1, \ldots , r_n~|~r_i^2,   i \in I, R(\Gamma) \rangle, ~\mbox{and }~L =
 \ll i^2,  i \in I, R(\Gamma) \gg$ is the normal subgroup of ${\bf I}$ which  is generated by the  elements  $ i^2$ for all $ i \in I$ and $ R(\Gamma) $ of ${\bf I}$.
 Further, let  
\begin{itemize}
\item $N$ be the  normal subgroup  of ${\bf I}$ generated by  $[L,L]$ and $ R(\Gamma) $, and   $V:= {\bf I}/N$. We denote the canonical  epimorphism  from ${\bf I}$ to $V$ by $q$ and we set $q_i:= iN$.
\item   Let $f$ be the  epimorphism $f: V \rightarrow W$
defined by $f(q_i) := r_i$, and  $U:= \ker f$.

\end{itemize}
 The homomorphism $f$ is well-defined since $N \subset L$, and we have $f\circ q = r$.  This setting implies the following.
 
\begin{lemma}\label{PropertiesU}
	Let $A= A(\Gamma)$ be the  Artin group of spherical type $\Gamma$. It follows that
	\begin{itemize}
		\item[(a)] $V/U \cong W$, and $U = L/N$;
		\item[(b)]  $U$ is an abelian normal subgroup of $V$;
		\item[(c)]  $U$ is the normal closure in $V$ of the words $q_i^2$  for all  $ i \in I$;
		\item[(d)]  $A= A(\Gamma) = {\bf I}/B$ where $B:= \langle\langle  R(\Gamma) \rangle\rangle \leq N \leq {\bf I}$, and $\PA := \PA(\Gamma) = L/B$;
		\item[(e)] $U$ is isomorphic to the abelianisation of the pure Artin group $\PA$.
	\end{itemize}
\end{lemma}
\begin{proof} Assertions (a) and (b) follow from the definition of $U$,  (c) is a consequence of the definition of $r$, and (d) follows
	from the definitions of the Artin and the pure Artin groups.
From (d), we conclude that $\PA/[\PA, \PA] = (L /B) / (([L,L]B / B) \cong L/([L,L] B) = L/N = U$, which shows that 
$U$ is isomorphic to the abelianisation of the pure Artin group $\PA$, that is (e).
\end{proof}

Tits determined the abelianisaton of the pure Artin group $\PA$ in 
 \cite[Theorem~2.5]{Tits}.
 
\begin{theorem}\label{Tits} (\cite[Theorem~2.5]{Tits})
There is a map $g: T \rightarrow U$ of the set of reflections of $W$ into the kernel $U$ of the map $f: V \rightarrow W$ such that for all $i \in I, t \in T$ and $w \in W$
$$q_i^2 = g(r_i)~\mbox{and}~ g(r_i^w) = g(r_i)^w.$$
These relations determine the map $g$. It is injective, and $U$ is the free abelian group generated by $g(T)$,  and is of rank $|T|$.
\end{theorem}

\subsection{The abelianisation of the pure interval group of type $D_n$}\label{SubNonIsomDn}

Now we consider  the interval group  $G := G_{m,n}$ related to the Coxeter group of type $D_n$, where $m > 1$. We adapt Tits notation, construction and arguments for Artin groups of spherical type
to  the group $G$ in order to prove an upper bound of the rank  of the abelianisation $K/[K,K]$ of the ``pure interval group'' $K$ of $G$ (see the definition 
below).  We keep the definitions of $I$ and ${\bf I}$, and let   $S = \{s_1, \ldots , s_n\}$ be the set of $n$ generators of  $W$ that correspond to the vertices of $\Delta_{m,n}$. Define a homomorphism 

\begin{itemize}
	\item $s: {\bf I} \rightarrow W$ by $s(i):= s_i$, and  let $L := \ker (s)$.
\end{itemize}
Then  we have 
$$W = \langle s_1, \ldots s_n ~|~s_i^2, i \in I,  R_{m,n},   \tc{s_1} {s_{m}}{s_{m+1}}{s_{m+2}} \rangle,$$
where $(s_1, s_{m}, s_{m+1}, s_{m+2})$  is  the unique $4$-cycle  in $\Delta_{m,n}$.
Let

\begin{itemize}
	\item $N =  \ll [L,L], R_{m,n}, \tc{1}{m}{m+1}{m+2} \gg \unlhd   ~  {\bf I}$ and $V:=V_{m,n}:=  {\bf I}/N$. We denote again  by $q$  the homomorphisms $q:{\bf I} \rightarrow V$ and set $q_i:= q(i)$.  
	\item Let $f$ be the epimorphism $f: V \rightarrow W$ defined by $f(q_i) = r_i$, and $U := \ker f$. 
\end{itemize}
We get properties analogous to those for the Artin groups in  Lemma~\ref{PropertiesU}:

\begin{lemma}\label{Basics}
	Let $G := G_{m,n}$ be the interval group of type $\Delta_{m,n}$.
It  follows that
	\begin{itemize}
		\item[(a)] $V/U \cong W$, and $U = L/N$;
		\item[(b)] $U$ is an abelian normal subgroup of $V$;
		\item[(c)]  $U$ is the normal closure of the words $q_i^2$ in $V$, where $ i \in I$;
		\item[(d)]  $G = {\bf I}/B$ where $B := \ll  R_{m,n}, \tc{1}{m}{m+1}{m+2} \gg \leq N \leq{\bf I}$ , and  $K:= L/B$ is the kernel of the map from $G$ to $W$ sending
		$\ss_i$ onto $s_i$;
		\item[(e)] $U$  is isomorphic to the abelianisation of $K$.
	\end{itemize}
\end{lemma}
\begin{proof}  The proofs of (a), (b) are identical to the proofs of Lemma~\ref{PropertiesU}(a), (b). The first part of assertion (d) is a consequence of  \cite[Theorem~A]{BaumNeaRees} and 
	the second part of the facts that $I/B \cong G$ and $I/L \cong W$ by (1) and the first part of (d).

To prove (c),
	let  $X$ be the normal closure of $q_i^2 \in V$ for  $1 \leq i \leq n$.  Clearly $X \subseteq L/N$. Observe that 
	$q_i X$ satisfies the relations  of the presentation for $W$ by \cite{Cameron}.  This shows $L/N \subseteq X$,
	and equality now follows. Statement (e) is immediate from the definitions of $L,N$ and $B$.
\end{proof}

In the following, we assume  that $n > 4$. This implies, as $m \leq n/2$, that $m +2 < n$. 
The next lemma is an important fact which holds in $A(D_n)$ (see \cite{Tits}), but also  in $G_{m,n}$ with $m >1$.

\begin{lemma}\label{ImportantFact}
We have that 
	\begin{itemize}
		\item[(a)] There is an  action of $W$ on $U$ given by  $w(u) = v(u)$
		 for $u \in U$ and $w \in W$,  where $v \in V$ is  any element 
		 such that  $f(v) = w$.
		
		\item[(b)]  $C_W(q_n^2) \geq C_W(s_n)$.
	\end{itemize}
\end{lemma}
\begin{proof} 
The action of $W$ on $U$ defined in (a) is well-defined  as  $U = \ker(f)$ and as $U$ is abelian. 
	It remains to prove (b).

	As $n \geq 5$, the element $q_n$ commutes with $q_i$ for $1 \leq i \leq n-2$.  
	In our notation, $s_n$ corresponds to the root $e_n - e_{n-1}$. Let 
	$$s := s_1^{s_{m+2} s_{m+3} \cdots s_ns_{m+1} \cdots s_{n-1}} \in W.$$
	Then $s$ corresponds to the root $e_n + e_{n-1}$ and $C_W(s_n) = \langle s_1, \ldots , s_{n-2}, s_n, s \rangle$.
	
	We note the following two properties.
	The first is an elementary calculation, and the second follows from Lemma~\ref{Basics} (c).
	\begin{itemize}
		\item[(i)] The braid relation $q_jq_kq_j=q_kq_jq_k$ implies $q_j^{-1} q_k^2 q_j = q_k q_j^2 q_k^{-1}$;
		\item[(ii)] $q_j^{-2} q_k^2 q_j^2 = q_k^2 $.
	\end{itemize}
	
	We apply these to derive $(q_n^2)^s = q_n^2$ by a  direct calculation.
	This then implies $C_W(q_n^2) \geq C_W(s_n)$.
\end{proof}

We give a second proof of Lemma~\ref{ImportantFact} (b), which does not use  any calculations. We will also apply the  strategy of the second proof to the exceptional groups.
\noindent\\
\begin{proof}

        We consider the two following  parabolic subgroups of $W$, the subgroups
        $$P:= \langle s_i, s_n ~|~ 1 \leq i \leq n-2 \rangle \cong  W(D_{n-2}) \times \Z_2~\mbox{and}$$
        $$P_e:=  \langle s_1, s_{m+1}, \ldots , s_n \rangle  \cong W(D_{n-m}), ~ (\mbox{see Figure~\ref{FigureCarterDiagramDn}}).$$
        Then  $\{s_1, s_{m+1}, \ldots , s_n\}$  is a simple system  in $P_e$, and
   $w_e:= s_1s_{m+1} \cdots s_n$  is  a parabolic Coxeter element in $W$, as well as   a prefix of $w$.
	It follows from \cite[Proposition~2.1(3)]{Tits}  
	       that $C_{P_e}(q^2_n) \geq  C_{P_e}(s_n) \cong W(D_{n-m-2} ) \times \Z_2^2$.  Moreover, as $q_i$ and $q_n$ commute for $1 \leq i \leq n-2$, the parabolic subgroup  $P $ fixes  $q_n^2$, from which we derive that
        $ \langle C_{P_e}(s_n) , P \rangle$ is a subgroup of $C_W(q_n^2)$. Now the claim follows from the fact that 
        $ \langle C_{P_e}(s_n) , P \rangle =  C_W(s_n)$ (which is isomorphic to $ W(D_{n-2} ) \times \Z_2^2$).
\end{proof}

\begin{cor}\label{Coro1}
	Let  $n \geq 5$. 
	Then the following hold:
	\begin{itemize}
		\item[(a)]  $q^2_1, \ldots , q^2_n$ are conjugate in $V$;
		\item[(b)]   ${\cal T}:= (q_i^2)^V = (q_k^2)^V$ for all $i,k \in I$;
	\end{itemize}
\end{cor}
\begin{proof} Assertion (a) is a consequence of the connectivity of the Carter diagram for $w$.
	Assertion (b) follows from (a).
\end{proof}

Next we show equality in Lemma~\ref{ImportantFact}(b), which implies that there is an
injective map $g:T \rightarrow U$, which sends $s_i$ to $q_i^2$.

\begin{lemma}\label{ImportantFactHelp}
We have  $C_W(q_n^2) = C_W(s_n)$ and $|{\cal T}| = |T|$.
\end{lemma}
\begin{proof}
We have that $C_W(q_n^2) \geq C_W(s_n)$ by Lemma~\ref{ImportantFact}, and therefore
$C_W(q_1^2) \geq C_W(s_1)$ by Lemma~\ref{Coro1}(a).
Recall that $W = O_2(W) \ltimes  \Sym(n)$, where $O_2(W)$ is the even part of the permutation module  for $\Sym(n)$, and that
$C:= C_W(s_1) = O\ltimes  H$, where $O = C_W(s_1) \cap O_2(W)$ is of order $2^{n-2}$ and $H \leq \Sym(n)$ is isomorphic to $\Sym(n-2) \times \Z_2$.
The group  $H$ is a maximal  in $\Sym(n)$.

First we show, if  $M$ is a proper overgroup of $C$  in $W$, then $O_2(W)$ is a subgroup of $M$.
Notice that  $C$ is a subgroup of index $2$ in $O_2(W)C$.  Assume $O_2(W) \not\leq M$.
Then, also  as $H$ is maximal in  $Sym(n)$, it follows, that $MO_2(W)/O_2(W) \cong \Sym(n)$. As $O_2(W)$ does not contain a
 $Sym(n)$-invariant subgroup of index $2$, we conclude $M = W$  and $O_2(W) \leq M$ in
 contradiction to our assumption.
 Thus, every overgroup of $C$ in $W$ contains  $O_2(W)$.

Consider the  subgroup $W_4$ of $W$ generated  by $s_1,s_m, s_{m+1}$ and $s_{m+2}$.
 It  is isomorphic to the Coxeter group of type $D_4$. Let $V_4$ be the subgroup
of $V$ that is generated by $q_1, q_m, q_{m+1}$ and $q_{m+2}$.
Further, let  $V_{2,4}$ be the group $V$ in the case that $(m, n) = (2,4)$. Then there is an epimorphism
$\varphi$ from $V_{2,4}$  onto $V_4$, which sends the elements $q_i$ in $V_{2,4}$, for $1 \leq i \leq 4$, onto the elements $q_1, ,q_m, q_{m+1}, q_{m+2}$ in $V$ according to the diagram.
Denote the kernel of $\varphi$ by $F$.
We derive from $f(V_4) \cong W(D_4)$  that $F$ is a subgroup of the kernel of the map
$f_{2,4}: V_{2,4} \rightarrow W(D_4)$, which sends $q_i$ onto $s_i$. As this kernel is abelian, the
subgroup $F$ commutes with $q_1^2$ in  $V_{2,4}$. This yields
$$C_{V_4}(q_1^2)  \cong C_{V_{2,4}/F}(q_1^2) =   C_{V_{2,4}}(q_1^2)  =  C_{W(D_4)}(q_1^2) = C_{W(D_4)}(s_1).$$
We checked the last equality with GAP.
We conclude  $C_{V_4}(q_1^2)  = C_{W_4}(s_1)$.

Assume that $C_W(q_1^2) >C =  C_W(s_1)$. Then $O_2(W) \leq C_W(q_1^2)$  by the second paragraph. Therefore, $\Z_2^3 \cong O_2(W) \cap W_4 \leq C_W(q_1^2)  \cap W_4 = C_{W_4}(q^2_1) = C_{W_4}(s_1) $,
which yields  the contradiction that $C_{W_4}(s_1)$, which is elementary abelian of order $16$, is a semi-direct product  of an elementary abelian group of order $8$ with an  elementary abelian group of order $4$. This proves  $C_W(q_n^2) = C_W(s_n)$. The second assertion is a consequence of the first.
\end{proof}

As a consequence  of Lemma~\ref{ImportantFactHelp}, we obtain that there is a bijective map $g$ from $T$ to  $\cal T
\subset U$ as in the case of a spherical Artin group; see Therorem~\ref{Tits}.

\begin{cor}\label{Coro2}
	Let  $n \geq 4$.  The  abelian group $U$ is generated by the set $\cal T$ of size $|T|$.
\end{cor}
\begin{proof} 
The	assertion  follows from Lemma~\ref{Basics} (c) and  Lemma~\ref{ImportantFactHelp}.
\end{proof}

In fact, a proper subgroup of $\cal T$ already  generates $U$.

\begin{proposition}\label{UpperBoundDn}
	Let $n \geq 4$.
	Then the abelianisation of the kernel of the map from $G_{m,n}$ to $W$ that takes $\ss_i$ to $s_i$  has a generating set of size at most $|T|-2$.
\end{proposition}
\begin{proof} We prove the assertion by induction on $n$.  We checked by hand and using GAP that $U$ is free abelian of rank $10$ if $n = 4$.
	Now let $n \geq 5$.  
	Let $P := \langle q_1, \ldots , q_{n-1} \rangle \leq V$. Then $P$ is a quotient of the 
	respective group $V_{k,n-1}$  related to $G_{k,n-1}$, where $k = m$  if $m <n/2$ and $k = m-1$   if $m = n/2$.
Let $T_{n-1}$ be the set of reflections in $f(P) \cong W(D_{n-1})$. 
	Then $|(q_1^2)^P|  = |T_{n-1}|$  and by induction $\langle (q_1^2)^P \rangle$ is abelian of  rank at most $|T_{n-1}|- 2$. This implies that 
	$$ U = \langle  {\cal T}  \rangle = 
	 \langle {\cal T}_{n-1} \rangle  \langle {\cal T} \setminus{ {\cal T}_{n-1}} \rangle $$
	is abelian of rank at most 
$ |T_{n-1}|  - 2 + (  |T|  -  |T_{n-1}| )=   |T| - 2$, as claimed. 
	\end{proof}

\subsection{Non-isomorphism}

\begin{theorem}\label{NonIsoDn}
	Let $W$ be a Coxeter group of type $D_n$ with $n\ge 4$ even, and $w$ a proper quasi-Coxeter element  in $W$. Then $G([1,w])$ is not isomorphic 
	to the Artin group of type  $D_n$.
\end{theorem}
\begin{proof}
	By Proposition~\ref{UpperBoundDn}, the abelianisation  of $L/B$ is of rank less than $|T|$ in $G_{m,n}$.  According to Theorem~\ref{CohenParis},
	there is a unique epimorphism form $A(D_n)$ to the Coxeter group $W(D_n)$ if $n$ is even. As according to Tits' Theorem~\ref{Tits} the abelianisation of the pure Artin group $\PA(D_n)$
 is of rank $|T|$, it follows that $G_{m,n}$ and $A(D_n)$ are not isomorphic.
\end{proof}

\subsection{The cases $E_6$, $E_7$, $E_8$}\label{SubNonIsomE6E7E8}

Now consider the exceptional cases. We use the same notation as before, and we just assume that $W$ is a Coxeter group of type $E_n$ for some $n \in \{6,7,8\}$.
Observe that Lemmas~\ref{Basics} and \ref{ImportantFact}(a) of the last section still hold. We need to show Lemma~\ref{ImportantFact}(b) case-by-case, which
we do next. Then analogous statements of the two Corollaries~\ref{Coro1} and \ref{Coro2} hold as well.


\begin{lemma} 
Let $W$ be a Coxeter group of type $E_n$ with $n \in \{6,7,8\}$, and let $w$ be a proper quasi-Coxeter element.
	Then there is $s_1 \in S$ such that $C_W(q_1^2) \geq C_W(s_1)$.
\end{lemma}
\begin{proof}
	We follow the second proof of Lemma~\ref{ImportantFact}(b). Our strategy is as follows.
	We have that 
	$$P:= \langle s_i ~|~ [s_i,s_1] = 1 \rangle~\mbox{as well as}~P_m:= \langle s_i ~|~ 1 \leq i  \leq   n, i \neq m\rangle $$
	are parabolic subgroups  of $W$ for every $m \in \{1, \ldots , n\}$ and $s_1 \in S$.
	       Therefore, we have $P \cap P_m = \langle s_i ~|~ [s_i,s_1] = 1 $  and  $s_i  \in P_m \rangle $
         by \cite[Chapter IV, Section 2]{Bourbaki}, which can be easily computed, as we will do in the following tables.

        In the three tables, we always choose $s_1$ to correspond to the vertex in the Carter diagram which is in the upper right corner and $s_m$ to be in the bottom right corner, except for Cases $E_7(a_3)$ and $E_8(a_8)$, where $s_m$ is in the bottom left corner, and we present $P$, $P_m$, $C_{P_m}(s_1)$, and $P \cap C_{P_m}(s_1) = P \cap P_m$. We describe a parabolic subgroup by writing down the type of the related root system.
        Our choice of $s_1$ is always such that $\langle P, C_{P_m}(s_1)  \rangle = C_W(s_1)$ for each $n \in \{6,7,8\}$ and each $E_n(a_i)$. By the definition of $P$, we have $P \leq C_W(q_1^2)$, and 
        by Lemma~\ref{ImportantFact}(b) and by induction we obtain $ C_{P_m}(q_1^2) \geq  C_{P_m}(s_1)$. This yields
         $$C_W(q^2_1) \geq \langle C_P(q_1^2), C_{P_m}(q_1^2)  \rangle \geq  \langle P, C_{P_m}(s_1)  \rangle = 
        C_W(s_1),$$ as claimed.
	\bigskip\\
	{\it Case $n = 6$:}
	\noindent
	$$\begin{array}{|c|c|c|c|c|}
		\hline
		\mbox{Type of}~w & P & P_m & C_{P_m}(s_1) & P \cap P_m \\\hline
		E_6(a_1) & A_1 \times A_4 & D_5  & A_1^2 \times D_3 & A_1 \times A_3 \\
		E_6(a_2) & A_1 \times A_3 & D_5  & A_1^2 \times D_3\ & A_1 \times A_3\\
		\hline
	\end{array}$$
	In the first case $P$ (resp. $C_{P_m}(s_1)$ in the second case) is a maximal subgroup in $C_W(s_1) \cong \Z_2 \times \Sym(6)$,
	which yields the assertion in both cases.
	\bigskip\\
	{\it Case $n = 7$:}
	\noindent
	$$\begin{array}{|c|c|c|c|c|}
		\hline
		\mbox{Type of}~w & P & P_m & C_{P_m}(s_1) & P \cap P_m \\\hline
		
		E_7(a_1) & A_1 \times D_5 & D_6 & A_1^2 \times D_4 & A_1 \times D_4\\
		E_7(a_2)  & A_1 \times A_5 & E_6 & A_1 \times A_5 &  A_1 \times A_4 \\
		E_7(a_3)  & A_1 \times D_4 & E_6 & A_1 \times A_5 & A_1 \times A_4 \\
		E_7(a_4)  & A_1 \times D_4 & E_6 & A_1 \times A_5 & A_1 \times  A_3\\
		\hline
	\end{array}$$
	
	If $W$ is of type $E_7$, then $C_W(s_1)$ is of type $A_1 \times D_6$, i.e. isomorphic to $\Z_2 \times \Z_{2}^5 \rtimes \Sym(6)$.
	The overgroups of $C_{P_m}(s_1) \cong \Z_2 \times \Sym(6)$ in $C_W(s_1)$  appearing in the table, if $a_i = a_2,a_3$ or $a_4$, are 
	       $$\Z_2 \times \Sym(6) < \Z_2^2 \times \Sym(6)  < \Z_2 \times \Z_{2}^5 \rtimes \Sym(6)= C_W(s_1).$$ 
	       Thus, if $a_i = a_2,a_3$ or $a_4$, then 
	       $P \cap P_m$ is not of index $2$ in $P$ and  we get $\langle P, C_{P_m}(s_1)  \rangle = C_W(s_1)$, thus the assertion.
 In the case $E_7(a_1)$, we construct the centraliser by hand, which we will do now.
	\medskip
	\\
	$E_7(a_1)$: Using Bourbaki's notation \cite{Bourbaki}, the following roots give an $E_7(a_1)$-diagram:
	$$e_1 + e_3, e_4 + e_1, e_5 - e_4,  e_6 - e_5$$
	$$e_3 - e_2, e_3 - e_1, 1/2(e_1+e_8) - 1/2(e_2+e_3+e_4 +e_5 +e_6 +e_7).$$
	Using the given roots we see that $C_{P_m}(s_1)$  is related to the root system
	generated by $e_6 - e_5,  e_5 + e_6,     e_1 + e_3, e_4 + e_1, e_3 - e_2, e_3 - e_1$.
	Thus $\langle P, C_{P_m}(s_1)  \rangle $ is generated by the reflections related 
	to the roots $e_6 - e_5,  e_5 + e_6,     e_1 + e_3, e_4 + e_1, e_3 - e_2, e_3 - e_1, 1/2(e_1+e_8) - 1/2(e_2+e_3+e_4 +e_5 +e_6 +e_7)$,
	which generate a root system of type $A_1 \times D_6$. Hence we get the assertion in this case as well.
	\bigskip\\
	{\it Case $n = 8$:}
	\noindent
	$$\begin{array}{|c|c|c|c|c|}
		\hline
		\mbox{Type of}~w & P & P_m & C_{P_m}(s_1) & P \cap P_m \\\hline
		E_8(a_1) & A_1 \times E_6 & D_7 & A_1^2 \times D_5 & A_1 \times D_5 \\
		E_8(a_2) & A_1 \times E_6 & D_7 & A_1^2 \times D_5 & A_1 \times D_5\\
		E_8(a_3) & A_1 \times E_6 & E_7 & A_1 \times D_6 & A_1 \times  A_5 \\
		E_8(a_4) & A_1 \times E_6 & D_7 & A_1^2 \times D_5 & A_1 \times  D_5\\
		E_8(a_5) & A_1 \times E_6 & D_7 & A_1^2 \times D_5 & A_1 \times  D_5\\
		E_8(a_6) & A_1 \times D_5 & E_7 & A_1 \times D_6 & A_1 \times  D_5\\
		E_8(a_7) & A_1 \times E_6 & E_7 & A_1 \times D_6 & A_1 \times  D_5\\
		E_8(a_8) & A_1 \times D_4 & E_7 & A_1 \times D_6 & A_1 \times  A_1^3\\
		\hline
	\end{array}$$
	
	Here $C_W(s_1)$ is of type $A_1 \times E_7$. In the cases where $P_m$ is of type $E_7$, the overgroups of $C_{P_m}(s_1)$ in $C_W(s_1)$
	        are $$C_{P_m}(s_1) = \Z_2 \times \Z_{2}^5 \rtimes \Sym(6) < \Z_2^2 \times \Z_{2}^5 \rtimes \Sym(6) < \Z_2 \times E_7 = C_W(s_1).$$
	It is straightforward to see that the index $|P : P \cap P_m| >2$ in all cases and thereby obtain $\langle P, C_{P_m}(s_1)  \rangle = C_W(s_1)$, which is the assertion.  Thus, it remains to consider the cases where $P_m$ is of type $D_7$. In these cases $P$ is a maximal subgroup of $C_W(s_1)$, but $C_{P_m}(s_1)$ is not contained in $P$, which also shows the assertion.
\end{proof}

\begin{theorem}\label{NonIsoEn}
	Let $W = W(E_n)$ be a Coxeter group of type $E_n$, $n \in \{6,7,8\}$, and $w$  a proper quasi-Coxeter element  in $W$. Then the interval group $G([1,w])$ is not isomorphic to the 
	Artin group $A(E_n)$.
\end{theorem}
\begin{proof}
	Let $w$ be a proper quasi-Coxeter element  in $W$, and let $P := \langle s_i ~|~ 2 \leq i  \leq  n\rangle $.  According to our setting, the Carter diagram of  $w$ where we remove the vertex related to $s_1$ contains a quadrangle. 
	Therefore, as in the proof of Proposition~\ref{UpperBoundDn}, we get by induction that the rank of the abelianisation $U$ is at most $|T|-2$ where $T$ is the set of reflections in $W$.
	Therefore, Theorems~\ref{CohenParis} and \ref{Tits} yield that $A(E_n)$ and $G([1,w])$ are not isomorphic.
\end{proof}

\section{Open questions}\label{SubNonIsonOpen}

Since each interval group related to a proper quasi-Coxeter element is not isomorphic to the corresponding Artin group, we will develop some open questions that are originally considered in the theory of Artin groups and Garside groups. A positive answer to all these questions exist for interval groups related to Coxeter elements (the case of Artin groups). So the questions are still open for the interval groups related to proper quasi-Coxeter elements.

\begin{itemize}
	\item[(a)] Can we solve the word and conjugacy problems for the interval groups?
	
	\item[(b)] Is the centre of each interval group infinite cyclic? Note that a certain power of the lift of the quasi-Coxeter element to the interval group is always central.
	
	\item[(c)] Are the interval groups torsion-free?
	
	\item[(d)] Is the monoid defined from the presentation in Proposition~\ref{PropDualPresDn} (viewed as a monoid presentation) cancellative? Does it inject in the corresponding interval group?
	
	\item[(e)] Can we describe the parabolic subgroups of the interval groups?
	
	\item[(f)] Is the interval complex related to the poset of non-crossing partitions of a proper quasi-Coxeter element a classifying space for the interval group? This question is relevant to the $K(\pi,1)$ conjecture for Artin groups.
	
\end{itemize}

\bibliographystyle{alpha}
\bibliography{IntervalGrpsII.bib}

%
%
%
%

\end{document}